\documentclass[a4paper,11pt,reqno]{amsart}

\usepackage[T1]{fontenc}
\usepackage{amsmath,amssymb,amsthm,mathtools,hyperref,geometry,cleveref,mathrsfs,xcolor,microtype,tabularx,array,multirow,float,orcidlink}

\allowdisplaybreaks

\definecolor{DarkBlue}{rgb}{0.1,0.1,0.55}
\definecolor{DarkRed}{rgb}{0.55,0.1,0.1}

\hypersetup{colorlinks=true,linkcolor=DarkBlue,citecolor=DarkRed,urlcolor=DarkBlue,pdftitle={Gaussian-Weighted Curvature Gaps for Self-Shrinkers},pdfauthor={Fagui Li and Yuhang Zhao}}

\makeatletter
\renewcommand{\@defaultbiblabelstyle}[1]{[#1]}
\makeatother

\numberwithin{equation}{section}

\newtheorem{theorem}{Theorem}[section]
\newtheorem{proposition}[theorem]{Proposition}
\newtheorem{lemma}[theorem]{Lemma}

\newtheorem{conjecture}[theorem]{Conjecture}
\theoremstyle{definition}

\theoremstyle{remark}
\newtheorem{remark}[theorem]{Remark}

\newcommand{\R}{\mathbb{R}}
\newcommand{\Sn}{\mathbb{S}}

\newcommand{\diver}{\operatorname{div}}

\newcommand{\inner}[2]{\left\langle #1,#2\right\rangle}

\newcommand{\Ent}{\Lambda}

\title[Gaussian-Weighted Curvature Gaps for Self-Shrinkers]{Gaussian-Weighted Curvature Gaps for Self-Shrinkers}

\author[F. Li]{Fagui Li}
\address{Frontier Interdisciplinary Domain, Beijing Institute of Technology, Zhuhai, Guangdong 519088, P. R. China}
\email{lifagui@bitzh.edu.cn}

\author[Y. Zhao]{Yuhang Zhao${}^{*}$}
\address{School of Mathematics, Nanjing University, Nanjing 210093, P. R. China}
\email{yuhangzhao@smail.nju.edu.cn}

\thanks{$^{*}$ The corresponding author.}
\thanks{F. Li is partially supported by NSFC (Nos. 12271040 and 12501061) and the Research Start-up Funding of Beijing Institute of Technology (No. 5640011253301).}
\begin{document}

\begin{abstract}
In this paper, we prove lower bounds for the Gaussian-weighted \(L^2\)-curvature integral of embedded self-shrinkers.
The proof combines normal coordinate functions with weighted Poincar\'e inequalities arising from first-eigenvalue estimates of Ding--Xin and Brendle--Tsiamis.  For closed self-shrinkers, the estimate gives an explicit lower bound in terms of entropy and, together with the entropy gap theorem of Colding--Ilmanen--Minicozzi--White, yields a strict curvature gap for nonspherical closed shrinkers.  In dimension two, we combine this estimate with the classification theorem and entropy gap theorem of Bernstein--Wang to obtain the corresponding gap statement for complete embedded self-shrinkers with polynomial volume growth.
\end{abstract}

\maketitle

\section{Introduction}

Self-shrinkers are basic self-similar solutions of the mean curvature flow and play a central role in the analysis of singularities.  We use the convention that an immersed hypersurface $M^n\subset\R^{n+1}$ is a self-shrinker if its mean curvature vector satisfies
\begin{equation}\label{eq:shrinker-vector}
 \mathbf H=-\frac{x^\perp}{2},
\end{equation}
where $x$ is the position vector and $x^\perp$ denotes the normal projection.  After a choice of unit normal, the scalar form of this equation may differ by a sign according to the convention for scalar mean curvature.  This sign will not affect the estimates below, since the argument uses only $|A|^2$, the Gaussian weight, and the drift Laplacian.

The foundational work of Huisken \cite{Huisken1990,Huisken1993} showed that self-shrinkers arise as tangent flows at type-I singularities.  Colding and Minicozzi \cite{ColdingMinicozzi2012Generic} introduced and systematically used the Gaussian area and entropy in their study of generic singularities of the mean curvature flow.  For a hypersurface $M^n\subset\R^{n+1}$, the Gaussian area is defined by
\begin{equation*}
 F_{x_0,t_0}(M)=(4\pi t_0)^{-\frac n2}
 \int_M e^{-\frac{|x-x_0|^2}{4t_0}}\,d\mu,
 \qquad x_0\in\R^{n+1},\quad t_0>0,
\end{equation*}
and the entropy
\begin{equation*}
 \Ent(M)=\sup_{x_0\in\R^{n+1},\,t_0>0}F_{x_0,t_0}(M).
\end{equation*}
Huisken's monotonicity formula implies that entropy is nonincreasing along the mean curvature flow.  If $M$ is a self-shrinker with polynomial volume growth, then the entropy is achieved at the Gaussian center and scale $(0,1)$; this is recalled precisely in Theorem~\ref{thm:CM-entropy} below.

Rigidity theorems for self-shrinkers have been studied from several different viewpoints.  One important direction concerns pointwise and integral pinching of the second fundamental form.  Le and Sesum \cite{LeSesum2011} proved a gap theorem for self-shrinking hypersurfaces.  Cao and Li \cite{CaoLi2013} established the sharp pointwise gap in arbitrary codimension.  Ding and Xin \cite{DingXin2014} derived Simons-type identities for self-shrinkers and proved several pointwise and integral rigidity results.  Ding \cite{Ding2018} later strengthened an integral rigidity theorem for the second fundamental form.  These results are closely related to the classical Simons identity and to curvature pinching for minimal submanifolds.

A second direction concerns entropy and low-complexity singularities.  Stone \cite{Stone1994} computed the entropy of the generalized cylinders $\Sn^k(\sqrt{2k})\times\R^{n-k}$.  Colding, Ilmanen, Minicozzi, and White \cite{ColdingIlmanenMinicozziWhite2013} proved that the round sphere minimizes entropy among closed self-shrinkers and that there is a positive entropy gap to the next closed self-shrinker.  In dimension two, Bernstein and Wang \cite{BernsteinWang2017} proved a classification and gap theorem for complete embedded self-shrinkers of small entropy, using in particular Brendle's \cite{Brendle2016} uniqueness theorem for genus-zero self-shrinkers.  Zhu \cite{Zhu2020} later proved the entropy lower bound for closed hypersurfaces in all dimensions and extended the entropy-stability classification to a singular setting.

Our argument is also motivated by the normal-coordinate-function method used by Ge and Li \cite{GeLi2022} for embedded hypersurfaces in the unit sphere.  In that setting, normal coordinate functions, an eigenvalue lower bound, and a divergence argument lead to a lower bound for the total squared norm of the second fundamental form.  Here we use the same guiding principle in the Gaussian geometry of self-shrinkers.  The spherical volume comparison is replaced by a Gaussian first-moment estimate, and the Choi--Wang eigenvalue estimate is replaced by the Brendle--Tsiamis drift-Laplacian estimate.

The purpose of this paper is to prove a weighted curvature estimate connecting entropy with the Gaussian $L^2$-norm of the second fundamental form.  The main observation, which is the Gaussian analogue of the normal-function argument in Ge--Li \cite{GeLi2022}, is as follows.  Let $a\in\R^{n+1}$ be a constant vector and set
\[
 \psi_a(x)=\inner{\nu(x)}{a}, \qquad x\in M^n,
\]
where $\nu$ is a chosen unit normal.  Then
\begin{equation}\label{eq:intro-normal-gradient}
 \nabla\psi_a=-A(a^T),
\end{equation}
where $a^T$ is the tangential projection of $a$ and $A$ is the shape operator with respect to $\nu$.  Summing \eqref{eq:intro-normal-gradient} over an ambient orthonormal basis gives exactly $|A|^2$.  Thus a weighted Poincar\'e inequality applied to the functions $\psi_a$ gives a lower bound for $\int_M |A|^2 
 e^{-\frac{|x|^2}{4}}
 \,d\mu$, once the weighted means of the functions $\psi_a$ are controlled.  We prove this control directly from the divergence theorem and the Gaussian first-moment estimate.

We now state the main closed case theorem.  Throughout, \(\Ent(\Sn^k)\) denotes the entropy of the self-shrinking round sphere \(\Sn^k(\sqrt{2k})\).

\begin{theorem}\label{thm:closed-main}
Let $M^n\subset\R^{n+1}$ be a connected, smooth, closed, embedded self-shrinker.  Then
\begin{equation}\label{eq:closed-main-entropy}
 \int_M |A|^2 
 e^{-\frac{|x|^2}{4}}
 \,d\mu
 \geq
 4^{\frac n2-1}\pi^{\frac n2}
 \frac{\Ent(M)^2-1}{\Ent(M)}.
\end{equation}
Consequently,
\begin{equation}\label{eq:closed-main-sphere}
 \int_M |A|^2 
 e^{-\frac{|x|^2}{4}}
 \,d\mu
 \geq
 4^{\frac n2-1}\pi^{\frac n2}
 \frac{\Ent(\Sn^n)^2-1}{\Ent(\Sn^n)}.
\end{equation}
Moreover, if $M$ is not the round self-shrinking sphere $\Sn^n(\sqrt{2n})$, then there exists a constant $\delta_n>0$, depending only on $n$, such that
\begin{equation}\label{eq:closed-main-gap}
 \int_M |A|^2 e^{-\frac{|x|^2}{4}}
 \,d\mu
 \geq
 4^{\frac n2-1}\pi^{\frac n2}
 \frac{(\Ent(\Sn^n)+\delta_n)^2-1}{\Ent(\Sn^n)+\delta_n}.
\end{equation}
\end{theorem}

The same argument also applies to complete embedded self-shrinkers with polynomial volume growth.  In dimension two, the entropy gap theorem of Bernstein and Wang \cite{BernsteinWang2017} gives the following form.

\begin{theorem}\label{thm:two-dimensional-gap}
Let $\Sigma^2\subset\R^3$ be a connected, complete,  embedded self-shrinker with polynomial volume growth.  Then
\begin{equation}\label{eq:surface-main}
 \int_\Sigma |A|^2 e^{-\frac{|x|^2}{4}}\,d\mu
 \geq
 \pi\,\frac{\Ent(\Sigma)^2-1}{\Ent(\Sigma)}.
\end{equation}
If $\Sigma$ is not a plane, then
\begin{equation}\label{eq:surface-sphere-lower}
 \int_\Sigma |A|^2 e^{-\frac{|x|^2}{4}}\,d\mu
 \geq
 \pi\,\frac{\Ent(\Sn^2)^2-1}{\Ent(\Sn^2)}.
\end{equation}
Furthermore, there exists a constant $\delta_0>0$ such that if $\Sigma$ is not a plane, not the round sphere $\Sn^2(2)$, and not the round cylinder $\Sn^1(\sqrt2)\times\R$, then
\begin{equation}\label{eq:surface-cylinder-gap}
 \int_\Sigma |A|^2 e^{-\frac{|x|^2}{4}}\,d\mu
 \geq
 \pi\,\frac{(\Ent(\Sn^1)+\delta_0)^2-1}{\Ent(\Sn^1)+\delta_0}.
\end{equation}
\end{theorem}
Colding, Ilmanen, Minicozzi, and White \cite{ColdingIlmanenMinicozziWhite2013} also formulated the following conjectural extension of their entropy gap theorem.

\begin{conjecture}[Colding--Ilmanen--Minicozzi--White]\label{conj:CIMW-nonflat}
This is Conjecture~0.10 in \cite{ColdingIlmanenMinicozziWhite2013}.
Let \(2\leq n\leq 6\).  There exists a constant
\(\epsilon_n>0\), depending only on \(n\), with the following property.
If \(M^n\subset \mathbb R^{n+1}\) is a smooth complete embedded self-shrinker
which is neither a hyperplane
nor the round sphere
\(\mathbb S^n(\sqrt{2n})\), then
\[
   \Ent(M)\geq \Ent(\mathbb S^n)+\epsilon_n .
\]
In particular, every non-flat such self-shrinker \(M^n\subset \mathbb R^{n+1}\) satisfies
\[
   \Ent(M)\geq \Ent(\mathbb S^n),
\]
and equality can occur only when $M$ is the round sphere
\(\mathbb S^n(\sqrt{2n})\).
\end{conjecture}


\begin{remark}[Conditional higher-dimensional extension]\label{rem:conditional-higher-dimensional}
Theorem~\ref{thm:two-dimensional-gap} is stated in dimension two because the required entropy lower bound for noncompact self-shrinkers is presently available in this form from Bernstein--Wang's theorem.  If Conjecture~\ref{conj:CIMW-nonflat} holds, then the non-flat part of Theorem~\ref{thm:two-dimensional-gap} extends to higher dimensions $3 \leq n\leq6$. 
For further related problems concerning entropy, spectral counting functions on shrinkers, codimension bounds, and low-entropy classifications in arbitrary codimension, see Colding--Minicozzi \cite[Section~0.4.1]{ColdingMinicozzi2019Complexity}.
\end{remark}

\begin{remark}\label{rem:sphere}
The estimates above are not sharp on the model shrinkers.  For instance, on the round self-shrinking sphere $\Sn^n(\sqrt{2n})$, one has $|A|^2=1/2$, and hence
\[
 \int_{\Sn^n(\sqrt{2n})}|A|^2\rho\,d\mu
 =
 \frac12(4\pi)^{\frac n2}\Ent(\Sn^n).
\]
The lower bound in Theorem~\ref{thm:closed-main} is smaller.  The point of the theorem is not sharpness of the constant, but rather that an entropy gap yields an explicit Gaussian $L^2$ curvature gap.
\end{remark}

Motivated by Theorems~\ref{thm:closed-main} and~\ref{thm:two-dimensional-gap}, and by Remarks  \ref{rem:conditional-higher-dimensional} and \ref{rem:sphere}, we propose the following sharp conjectures.

\begin{conjecture}\label{conj:closed-hypersurface}
Let $M^n\subset\R^{n+1}$ be a connected, smooth, closed, embedded self-shrinker.  Then
\[
 \int_M |A|^2 e^{-\frac{|x|^2}{4}}\,d\mu
 \geq
 \frac12(4\pi)^{\frac n2}\Ent(M).
\]
Equality holds if and only if $M$ is the round self-shrinking sphere
$\Sn^n(\sqrt{2n})$.
\end{conjecture}

\begin{conjecture}\label{conj:complete-hypersurface}
Let $M^n\subset\R^{n+1}$ be a connected, complete,  embedded self-shrinker with polynomial volume growth.  If $M$ is not a hyperplane, then
\[
 \int_M |A|^2 e^{-\frac{|x|^2}{4}}\,d\mu
 \geq
 \frac12(4\pi)^{\frac n2}\Ent(M).
\]
Equality holds if and only if $M$ is a generalized round cylinder of the form
$
 \Sn^k(\sqrt{2k})\times\R^{n-k}
$
for some $1\leq k\leq n$.
\end{conjecture}

The paper is organized as follows.  In Section~\ref{sec:weighted}, we prove the weighted normal-coordinate estimates and record the external results used later.  In Section~\ref{sec:proofs}, we prove Theorems~\ref{thm:closed-main} and~\ref{thm:two-dimensional-gap}.

\section{Preliminaries and weighted normal-coordinate estimates}\label{sec:weighted}

Let $M^n\subset\R^{n+1}$ be a connected, properly embedded, two-sided hypersurface.  We assume that $M$ separates $\R^{n+1}$ into two open components, denoted by $\Omega_+$ and $\Omega_-$.  This holds for the closed embedded hypersurfaces in Theorem~\ref{thm:closed-main}; it also holds for the  complete embedded hypersurfaces considered in Theorem~\ref{thm:two-dimensional-gap} from the work  of Cheng and Zhou \cite[Theorem~1.3]{ChengXuZhouDetang}.  We choose a global unit normal $\nu$ and put
\[
 \rho=e^{-\frac{|x|^2}{4}}.
\]
We use $d\mu$ for the induced $n$-dimensional measure on $M$ and on boundary components identified with $M$ and $dx$ for Lebesgue measure on $\R^{n+1}$. Let $B_R(0)=\{x\in\mathbb{R}^{n+1}: |x|\leq R\}$ for any  $R>0$.  We denote by $\bar\nabla$ the Euclidean connection.
For a constant vector $a\in\R^{n+1}$ define
\begin{equation*}
 \varphi_a(x)=\inner{x}{a},
 \qquad
 \psi_a(x)=\inner{\nu(x)}{a}.
\end{equation*}

\begin{lemma}\label{lem:grad-identities}
For every constant vector $a\in\R^{n+1}$,
\begin{equation*}
 \nabla\varphi_a=a^T,
 \qquad
 \nabla\psi_a=-A(a^T),
\end{equation*}
where $a^T$ denotes the tangential projection of $a$.
\end{lemma}

\begin{proof}
Let $X$ be a tangent vector field on $M$.  Since $a$ is constant in the ambient Euclidean space,
\[
 X(\varphi_a)=X\inner{x}{a}=\inner{X}{a}=\inner{X}{a^T},
\]
which gives $\nabla\varphi_a=a^T$.  For $\psi_a$, using the Weingarten equation $\bar\nabla_X\nu=-A(X)$, we obtain
\[
 X(\psi_a)=X\inner{\nu}{a}=\inner{\bar\nabla_X\nu}{a}
 =-\inner{A(X)}{a^T}=-\inner{X}{A(a^T)},
\]
because $A$ is self-adjoint.  Hence $\nabla\psi_a=-A(a^T)$.
\end{proof}

The following estimate is the only place where the separation property is used.

\begin{lemma}\label{lem:weighted-mean-bound}
Assume that $M$ has polynomial volume growth.  Then for every vector $a\in\R^{n+1}$ with $|a|=1$,
\begin{equation}\label{eq:weighted-mean-bound}
 \left|\int_M \psi_a\rho\,d\mu\right|
 \leq 2^n\pi^{\frac n2}.
\end{equation}
\end{lemma}

\begin{proof}
Let $\eta_\pm$ be the outward unit normal to $\Omega_\pm$ along $M$.  Since $\eta_\pm=\pm\nu$, it is enough to estimate
\[
 \left|\int_{\partial\Omega_\pm}\inner{a}{\eta_\pm}\rho\,d\mu\right|.
\]
Fix one component $\Omega=\Omega_\pm$ ($\eta=\eta_\pm$) and for any $R>0$, let $\phi_R$ be a smooth cut-off function with compact support in $B_{R+1}(0)$ such that 
$$  \phi_R= \begin{cases}
     1,&\textup{in} \ B_R(0) \\
 0,&\textup{in}\  \mathbb{R}^{n+1}\setminus B_{R+1}(0).
  \end{cases}$$
and $0\leq \phi_R \leq 1, |\bar\nabla \phi_R| \leq 2$.

The divergence theorem applied to $\Omega$ gives
\begin{align*}
\int_{M\cap B_{R+1}}\phi_R \inner{a}{\eta}\rho\,d\mu
=\int_{\Omega \cap B_{R+1}}\diver(\phi_R \rho a)\,dx.
\end{align*}
Since
\begin{align*}
 \diver(\phi_R\rho a)&=\phi_R \inner{\bar\nabla\rho}{a}+\rho \inner{\bar \nabla \phi_R}{a}\\
 &=-\frac12 \phi_R \inner{x}{a}\rho+\rho \inner{\bar \nabla \phi_R}{a},
\end{align*}
we have
\begin{equation*}
\left|\int_{M\cap B_{R+1}}\phi_R \inner{a}{\eta}\rho\,d\mu\right|
\leq \frac12\int_{\Omega\cap B_{R+1}}  |\inner{x}{a}|\rho\,dx
+2\int_{\Omega\cap \big(B_{R+1}\setminus B_R \big)}\rho\,dx,
\end{equation*}
where we use the definition  of $\phi_R$.

The polynomial volume growth of $M$ implies $\int_M\rho\,d\mu<\infty$.  Since $|\inner{a}{\eta}|\leq 1$, dominated convergence gives convergence of the boundary integrals over $M\cap B_R$ to the corresponding integral over $\partial\Omega$, with $\partial\Omega$ identified with $M$.  Letting $R\to\infty$ and using the monotone convergence theorem for the right-hand of the above inequality, we obtain
\begin{equation}\label{eq:component-estimate}
\left|\int_{\partial\Omega}\inner{a}{\eta}\rho\,d\mu\right|
\leq \frac12\int_{\Omega}|\inner{x}{a}|\rho\,dx.
\end{equation}
The two components $\Omega_+$ and $\Omega_-$ partition $\R^{n+1}$ up to the hypersurface $M$, which has zero $(n+1)$-dimensional measure.  Hence one component, say $\Omega_\sigma$, satisfies
\begin{equation*}
 \int_{\Omega_\sigma}|\inner{x}{a}|\rho\,dx
 \leq \frac12\int_{\R^{n+1}}|\inner{x}{a}|\rho\,dx.
\end{equation*}
For this chosen component, the outward unit normal along $M$ is either $\nu$ or $-\nu$; therefore the absolute value of the corresponding boundary integral is exactly $\left|\int_M\psi_a\rho\,d\mu\right|$.  Using this component in \eqref{eq:component-estimate} gives
\begin{equation}\label{eq:before-gaussian-moment}
\left|\int_M \psi_a\rho\,d\mu\right|
\leq \frac14\int_{\R^{n+1}}|\inner{x}{a}|e^{-|x|^2/4}\,dx.
\end{equation}
By rotational invariance, it suffices to take $a=e_1$, where $e_1=(1,0,\ldots,0)$ is the first standard basis vector of $\R^{n+1}$.  Then
\begin{equation}\label{eq:gaussian-first-moment}
\begin{aligned}
\int_{\R^{n+1}}|\inner{x}{a}|e^{-|x|^2/4}\,dx
&=\left(\int_{-\infty}^{\infty}|s|e^{-s^2/4}\,ds\right)
  \left(\int_{-\infty}^{\infty}e^{-t^2/4}\,dt\right)^n \\
&=4(2\sqrt\pi)^n=2^{n+2}\pi^{\frac n2}.
\end{aligned}
\end{equation}
Combining \eqref{eq:before-gaussian-moment} and \eqref{eq:gaussian-first-moment} yields \eqref{eq:weighted-mean-bound}.
\end{proof}
\begin{remark}
The constants depend on the normalization $\rho=e^{-|x|^2/4}$, which is the standard weight for the shrinker equation \eqref{eq:shrinker-vector}.  If one changes the Gaussian weight, the first-moment computation in \eqref{eq:gaussian-first-moment} and the final constants must be rescaled accordingly.
\end{remark}
\begin{lemma}\label{lem:basis}
Assume that $M$ has polynomial volume growth.  Then there exists an ambient orthonormal basis $\{a_1,\ldots,a_{n+1}\}$ of $\R^{n+1}$ such that
\begin{equation*}
 \int_M \psi_{a_i}\rho\,d\mu=0,
 \qquad 1\leq i\leq n.
\end{equation*}
\end{lemma}

\begin{proof}
The vector
\[
 V=\int_M \nu\rho\,d\mu\in\R^{n+1}
\]
is well defined because $|\nu|=1$ and $\int_M\rho\,d\mu<\infty$.  For every $a\in\R^{n+1}$,
\[
 \int_M\psi_a\rho\,d\mu=\inner{V}{a}.
\]
If $V\ne0$, choose $a_{n+1}=V/|V|$ and choose $a_1,\ldots,a_n$ as an orthonormal basis of $V^\perp$.  If $V=0$, every orthonormal basis has the desired property.
\end{proof}

We recall the external results that will be used in the proof below.  They are stated explicitly for the reader's convenience.

\begin{theorem}[Brendle--Tsiamis \cite{BrendleTsiamis2024}; Ding--Xin \cite{DingXin2013}]\label{thm:BT}
Let $M^n\subset\R^{n+1}$ be a complete embedded self-shrinker with polynomial volume growth.  Suppose that $f\in H^1_{\mathrm{loc}}(M)$ satisfies
\[
 \int_M e^{-\frac{|x|^2}{4}}(f^2+|\nabla f|^2)\,d\mu<\infty,
 \qquad
 \int_M f e^{-\frac{|x|^2}{4}}\,d\mu=0.
\]
Then
\begin{equation*}
 \int_M |\nabla f|^2 e^{-\frac{|x|^2}{4}}\,d\mu
 \geq \frac14\int_M f^2 e^{-\frac{|x|^2}{4}}\,d\mu.
\end{equation*}
\end{theorem}

The noncompact  embedded case is due to Brendle and Tsiamis \cite[Theorem~1]{BrendleTsiamis2024}; the compact case was proved earlier by Ding and Xin \cite[Theorem~1.3]{DingXin2013}.  The estimate is the shrinker version  of the first eigenvalue estimate of Choi and Wang \cite{ChoiWang1983} for embedded minimal hypersurfaces in spheres, which was later used by Choi and Schoen \cite{ChoiSchoen1985} in their compactness theorem for minimal surfaces.

For comparison, we recall the first eigenvalue conjecture of Yau \cite{Yau1982}; see also Schoen--Yau \cite{SchoenYau1994}.  The conjecture asserts that if $M^m\subset \Sn^{m+1}(1)$ is a closed embedded minimal hypersurface in the unit sphere, then the first nonzero eigenvalue of the Laplace--Beltrami operator on $M$ satisfies $\lambda_1(M)=m$.  Although the conjecture remains a central problem in general, it has been verified in several important cases.  Tang and Yan \cite{TangYan2013} and Tang--Xie--Yan \cite{TangXieYan2014} proved the conjecture for isoparametric minimal hypersurfaces.  Quantitative improvements of the Choi--Wang lower bound have also been obtained recently: Duncan--Sire--Spruck \cite{DuncanSireSpruck2024} proved an explicit estimate of the form $\lambda_1(M)\ge m/2+\varepsilon$ with $\varepsilon>0$ depending on the dimension and an upper bound for $|A|$, and Jim\'enez--Chinchay--Zhou \cite{JimenezChinchayZhou2026} obtained a further lower bound in terms of the dimension and $\max_M |A|$.  Zhao \cite{Zhao2026FirstEigenvalue} proved that
\[
  \lambda_1(M)>\frac m2+G(m,\max_M|A|,\min_M|A| ),
\]
where $G$ is a positive function depending only on $m$, $\max_M|A|$, and $\min_M|A|$; in particular, when $|A|$ is constant, this gives a dimension-dependent gap above $m/2$.  Thus Theorem~\ref{thm:BT} may be viewed as the Gaussian self-shrinker counterpart of this circle of first-eigenvalue estimates.

\begin{theorem}[Colding--Minicozzi \cite{ColdingMinicozzi2012Generic}]\label{thm:CM-entropy}
Let $M^n\subset\R^{n+1}$ be a self-shrinker with polynomial volume growth.  Then the entropy is achieved at center $0$ and scale $1$; namely,
\begin{equation}\label{eq:CM-entropy}
 \Ent(M)=F_{0,1}(M)=(4\pi)^{-\frac n2}\int_M e^{-\frac{|x|^2}{4}}\,d\mu.
\end{equation}
\end{theorem}

This is part of the entropy theory developed by Colding and Minicozzi \cite{ColdingMinicozzi2012Generic}.  The polynomial volume growth assumption guarantees the finiteness of the Gaussian area in \eqref{eq:CM-entropy}.

\begin{lemma}[Stone \cite{Stone1994}]\label{lem:stone}
For the self-shrinking generalized cylinders
\[
 \Sn^k(\sqrt{2k})\times\R^{n-k}\subset\R^{n+1},
 \qquad 1\leq k\leq n,
\]
the entropy is independent of the Euclidean factor.  More precisely,
\[
 \Ent\bigl(\Sn^k(\sqrt{2k})\times\R^{n-k}\bigr)
 =
 \Ent\bigl(\Sn^k(\sqrt{2k})\bigr).
\]
Moreover,
\[
 2>\Ent\bigl(\Sn^1(\sqrt2)\bigr)>
 \Ent\bigl(\Sn^2(2)\bigr)>
 \cdots>
 \Ent\bigl(\Sn^n(\sqrt{2n})\bigr)>
 \cdots>1 .
\]
The lower endpoint is the entropy of the flat hyperplane, namely,
$
 \Ent(\R^n)=1.
$
\end{lemma}

\begin{theorem}[Colding--Ilmanen--Minicozzi--White \cite{ColdingIlmanenMinicozziWhite2013}]\label{thm:CIMW}
For each $n\geq1$, the round sphere $\Sn^n(\sqrt{2n})$ minimizes entropy among all smooth, closed, embedded self-shrinkers in $\R^{n+1}$.  Moreover, there exists $\delta_n>0$ such that if $M^n\subset\R^{n+1}$ is a smooth, closed, embedded self-shrinker which is not the round sphere $\Sn^n(\sqrt{2n})$, then
\begin{equation*}
 \Ent(M)\geq \Ent(\Sn^n)+\delta_n.
\end{equation*}
\end{theorem}

For hypersurfaces, Bernstein--Wang proved this sharp entropy lower bound for all smooth, closed, embedded hypersurfaces up to dimension six \cite{BernsteinWangLuInvent}; the dimension restriction was later removed by Zhu \cite{Zhu2020}.

We will also use the following two-dimensional entropy gap.

\begin{theorem}[Bernstein--Wang \cite{BernsteinWang2017}]\label{thm:BW}
Let $\Sigma^2\subset\R^3$ be a smooth, complete embedded self-shrinker with polynomial volume growth.  If
\begin{equation*}
 \Ent(\Sigma)\leq \Ent(\Sn^1),
\end{equation*}
then $\Sigma$ is one of the following: the round sphere $\Sn^2(2)$, a plane through the origin, or a round cylinder $\Sn^1(\sqrt2)\times\R$ up to a rotation.  In addition, there exists $\delta_0>0$ such that these are the only complete embedded self-shrinkers in $\R^3$ with entropy less than $\Ent(\Sn^1)+\delta_0$.
\end{theorem}

The proof of Theorem \ref{thm:BW} (see Corollary 1.2 in \cite{BernsteinWang2017}) combines a topological theorem for asymptotically conical self-shrinkers with Brendle's \cite{Brendle2016} classification of embedded genus-zero self-shrinkers.

\section{Proof of the theorems}\label{sec:proofs}
Following the normal-coordinate-function idea used by Ge and Li \cite{GeLi2022}, we prove the weighted estimate below.  The key point is to combine the weighted Poincar\'e inequality with normal coordinate functions and the Gaussian divergence formula.  If the weighted curvature integral is infinite, then the desired lower bound is automatic; hence, throughout the proof of Proposition~\ref{prop:general}, we may assume that
\[
\int_M |A|^2\rho\,d\mu<\infty.
\]

\begin{proposition}\label{prop:general}
Let $M^n\subset\R^{n+1}$ be a connected, properly embedded, two-sided hypersurface with polynomial volume growth.  Assume that $M$ separates $\R^{n+1}$ into two components.  Suppose that there exists a constant $\lambda>0$ such that, for every smooth function $f$ satisfying
\[
 \int_M (f^2+|\nabla f|^2)\rho\,d\mu<\infty,
 \qquad
 \int_M f\rho\,d\mu=0,
\]
one has the weighted Poincar\'e inequality
\begin{equation}\label{eq:abstract-poincare}
 \int_M |\nabla f|^2\rho\,d\mu
 \geq \lambda\int_M f^2\rho\,d\mu.
\end{equation}
Then
\begin{equation}\label{eq:general-estimate}
 \int_M |A|^2\rho\,d\mu
 \geq
 \lambda\,\frac{\left(\int_M\rho\,d\mu\right)^2-4^n\pi^n}{\int_M\rho\,d\mu}.
\end{equation}
\end{proposition}

\begin{proof}
If $\int_M |A|^2\rho\,d\mu=\infty$, then \eqref{eq:general-estimate} is automatic.  Hence we assume throughout the proof that this integral is finite.  Then each $\psi_a$ satisfies
\[
 \int_M (\psi_a^2+|\nabla\psi_a|^2)\rho\,d\mu<\infty,
\]
because $|\psi_a|\leq |a|$ and $|\nabla\psi_a|\leq |A||a|$ by Lemma~\ref{lem:grad-identities}.

Choose the orthonormal basis $\{a_1,\ldots,a_{n+1}\}$ from Lemma~\ref{lem:basis}.  Applying \eqref{eq:abstract-poincare} to $\psi_{a_i}$ for $1\leq i\leq n$, and applying it to the zero-mean function $\psi_{a_{n+1}}-c$, where $c=(\int_M\psi_{a_{n+1}}\rho\,d\mu)/(\int_M\rho\,d\mu)$, gives
\begin{equation}\label{eq:poincare-sum}
\begin{aligned}
\sum_{\alpha=1}^{n+1}\int_M |\nabla\psi_{a_\alpha}|^2\rho\,d\mu
&\geq \lambda\sum_{\alpha=1}^{n+1}\int_M\psi_{a_\alpha}^2\rho\,d\mu \\
&\quad -\lambda\frac{\left(\int_M\psi_{a_{n+1}}\rho\,d\mu\right)^2}{\int_M\rho\,d\mu}.
\end{aligned}
\end{equation}
Here the weighted means of $\psi_{a_1},\ldots,\psi_{a_n}$ vanish by construction.

We next identify the two sums in \eqref{eq:poincare-sum}.  Since $\{a_\alpha\}_{\alpha=1}^{n+1}$ is an orthonormal basis of $\R^{n+1}$,
\begin{equation}\label{eq:sum-psi}
 \sum_{\alpha=1}^{n+1}\psi_{a_\alpha}^2
 =\sum_{\alpha=1}^{n+1}\inner{\nu}{a_\alpha}^2=|\nu|^2=1.
\end{equation}
By Lemma~\ref{lem:grad-identities},
\[
 \sum_{\alpha=1}^{n+1}|\nabla\psi_{a_\alpha}|^2
 =\sum_{\alpha=1}^{n+1}|A(a_\alpha^T)|^2.
\]
Let $\{e_1,\ldots,e_n\}$ be a local orthonormal tangent frame.  Since $A(e_i)$ is tangent,
\begin{equation}\label{eq:sum-grad-S}
\begin{aligned}
\sum_{\alpha=1}^{n+1}|A(a_\alpha^T)|^2
&=\sum_{\alpha=1}^{n+1}\sum_{i=1}^n\inner{A(a_\alpha^T)}{e_i}^2 \\
&=\sum_{i=1}^n\sum_{\alpha=1}^{n+1}\inner{a_\alpha}{A(e_i)}^2
 =\sum_{i=1}^n|A(e_i)|^2=|A|^2.
\end{aligned}
\end{equation}
Combining \eqref{eq:poincare-sum}, \eqref{eq:sum-psi}, and \eqref{eq:sum-grad-S}, we obtain
\begin{equation}\label{eq:before-mean-bound}
 \int_M |A|^2\rho\,d\mu
 \geq
 \lambda\left(\int_M\rho\,d\mu-
 \frac{\left(\int_M\psi_{a_{n+1}}\rho\,d\mu\right)^2}{\int_M\rho\,d\mu}\right).
\end{equation}
Finally, Lemma~\ref{lem:weighted-mean-bound} gives
\[
 \left(\int_M\psi_{a_{n+1}}\rho\,d\mu\right)^2\leq 4^n\pi^n.
\]
Substitution into \eqref{eq:before-mean-bound} proves \eqref{eq:general-estimate}.
\end{proof}

We now prove the two announced gap statements.

\begin{proof}[Proof of Theorem~\ref{thm:closed-main}]
For a closed embedded self-shrinker, Theorem~\ref{thm:BT} gives \eqref{eq:abstract-poincare} with $\lambda=1/4$.  Proposition~\ref{prop:general} therefore gives
\begin{equation}\label{eq:closed-step1}
 \int_M |A|^2\rho\,d\mu
 \geq
 \frac14\,\frac{\left(\int_M\rho\,d\mu\right)^2-4^n\pi^n}{\int_M\rho\,d\mu}.
\end{equation}
For  closed self-shrinkers, Theorem~\ref{thm:CM-entropy} gives
\begin{equation}\label{eq:int-rho-entropy}
 \int_M\rho\,d\mu=(4\pi)^{\frac n2}\Ent(M).
\end{equation}
Substituting \eqref{eq:int-rho-entropy} into \eqref{eq:closed-step1} gives
\begin{align*}
\frac14\frac{\left(\int_M\rho\,d\mu\right)^2-4^n\pi^n}{\int_M\rho\,d\mu}
&=\frac14\frac{(4\pi)^n\Ent(M)^2-4^n\pi^n}{(4\pi)^{\frac n2}\Ent(M)} \\
&=\frac14(4\pi)^{\frac n2}\frac{\Ent(M)^2-1}{\Ent(M)} \\
&=4^{\frac n2-1}\pi^{\frac n2}\frac{\Ent(M)^2-1}{\Ent(M)}.
\end{align*}
This proves \eqref{eq:closed-main-entropy}.

The function
\begin{equation*}
 \Phi(t)=\frac{t^2-1}{t}=t-\frac1t,
 \qquad t>0,
\end{equation*}
is strictly increasing because $\Phi'(t)=1+t^{-2}>0$.  By Theorem~\ref{thm:CIMW}, $\Ent(M)\geq\Ent(\Sn^n)$.  Applying the monotonicity of $\Phi$ to \eqref{eq:closed-main-entropy} gives \eqref{eq:closed-main-sphere}.  If $M$ is not the round sphere, Theorem~\ref{thm:CIMW} gives $\Ent(M)\geq\Ent(\Sn^n)+\delta_n$, and the same monotonicity gives \eqref{eq:closed-main-gap}.
\end{proof}

\begin{proof}[Proof of Theorem~\ref{thm:two-dimensional-gap}]
The polynomial volume growth assumption implies finite Gaussian area.  Theorem~\ref{thm:BT} again gives the weighted Poincar\'e inequality with constant $1/4$.  Applying Proposition~\ref{prop:general} with $n=2$ yields
\begin{equation}\label{eq:surface-step1}
 \int_\Sigma |A|^2\rho\,d\mu
 \geq
 \frac14\frac{\left(\int_\Sigma\rho\,d\mu\right)^2-16\pi^2}{\int_\Sigma\rho\,d\mu}.
\end{equation}
By Theorem~\ref{thm:CM-entropy},
\begin{equation*}
 \int_\Sigma\rho\,d\mu=(4\pi)\Ent(\Sigma).
\end{equation*}
Substitution into \eqref{eq:surface-step1} gives
\begin{align*}
\frac14\frac{(4\pi)^2\Ent(\Sigma)^2-16\pi^2}{4\pi\Ent(\Sigma)}
&=\pi\frac{\Ent(\Sigma)^2-1}{\Ent(\Sigma)}.
\end{align*}
This proves \eqref{eq:surface-main}.

If $\Sigma$ is not a plane, then Theorem~\ref{thm:BW} and Stone's ordering in Lemma~\ref{lem:stone} imply
\begin{equation}\label{eq:nonflat-entropy-lower}
 \Ent(\Sigma)\geq\Ent(\Sn^2).
\end{equation}
Indeed, if $\Ent(\Sigma)<\Ent(\Sn^2)$, then since $\Ent(\Sn^2)<\Ent(\Sn^1)$ by Lemma~\ref{lem:stone}, Theorem~\ref{thm:BW} would force $\Sigma$ to be a plane, a sphere, or a cylinder.  The plane is excluded by the hypothesis, while the sphere and cylinder have entropies $\Ent(\Sn^2)$ and $\Ent(\Sn^1)$, respectively, a contradiction.  Hence \eqref{eq:nonflat-entropy-lower} holds.  Since $t\mapsto (t^2-1)/t$ is increasing, \eqref{eq:surface-main} gives \eqref{eq:surface-sphere-lower}.

Finally, if $\Sigma$ is not a plane, not the round sphere, and not the round cylinder, Theorem~\ref{thm:BW} gives
\begin{equation*}
 \Ent(\Sigma)\geq \Ent(\Sn^1)+\delta_0.
\end{equation*}
Using again the monotonicity of $t-t^{-1}$ in \eqref{eq:surface-main} proves \eqref{eq:surface-cylinder-gap}.
\end{proof}

\bibliographystyle{amsplain}
\bibliography{260610selfreferences}

\end{document}